\documentclass[11pt,leqno]{amsart}
\usepackage{amssymb}
\usepackage{amsfonts}
\usepackage{amsthm}
\usepackage{mathtools}
\usepackage{todonotes}

\calclayout

\usepackage{amsthm}

\usepackage[pagebackref,colorlinks,citecolor=blue,linkcolor=blue]{hyperref}

\usepackage{cite}   

\usepackage{amsmath,bbm,amssymb,amsxtra}

\def\Xint#1{\mathchoice
	{\XXint\displaystyle\textstyle{#1}}
	{\XXint\textstyle\scriptstyle{#1}}
	{\XXint\scriptstyle\scriptscriptstyle{#1}}
	{\XXint\scriptscriptstyle\scriptscriptstyle{#1}}
	\!\int}

\def\XXint#1#2#3{{\setbox0=\hbox{$#1{#2#3}{\int}$ }
		\vcenter{\hbox{$#2#3$ }}\kern-.6\wd0}}
\def\dashint{\Xint-}

\theoremstyle{definition}
\newtheorem{theorem}{Theorem}[section]
\newtheorem{lemma}[theorem]{Lemma}
\newtheorem{proposition}[theorem]{Proposition}
\newtheorem{remark}[theorem]{Remark}
\newtheorem{corollary}[theorem]{Corollary}

\newtheorem{example}[theorem]{Example}

\numberwithin{equation}{section}

\newcommand{\N}{{\mathbb N}}
\newcommand{\R}{{\mathbb R}}
\newcommand{\liploc}{\mathrm{Lip}_{\mathrm{loc}}}
\newcommand{\BV}{\mathrm{BV}}
\newcommand{\loc}{\mathrm{loc}}

\DeclareMathOperator{\dist}{dist}
\DeclareMathOperator{\supp}{supp}

\title[Lusin approximation  for
functions of bounded variation]{Lusin approximation  for\\
	functions of bounded variation}

\author{Panu Lahti}\address{Academy of Mathematics and Systems Science, Chinese Academy of Sciences, Beijing 100190, PR China}\email{panulahti@amss.ac.cn}

\author{Khanh Nguyen}\address{Academy of Mathematics and Systems Science, Chinese Academy of Sciences, Beijing 100190, PR China}\email{khanhnguyen@amss.ac.cn}


\usepackage{framed}
\usepackage[most]{tcolorbox}
\usepackage{lipsum}

\definecolor{formalshade}{rgb}{0.95,0.95,1}
\newtcolorbox{mybox}{
	enhanced,
	boxrule=0pt,
	frame hidden,
	borderline west={1pt}{0pt}{},
	borderline east={1pt}{0pt}{},
	colback=white,
	sharp corners, 
	fontupper=\itshape
}
\begin{document}
	\maketitle
	
	\begin{center}
		{ \today}
	\end{center}
	
	\begin{abstract}
		We prove a Lusin approximation of functions of bounded variation.
		If $f$ is a function of  bounded  variation on an open set $\Omega\subset X$,
		where $X=(X,d,\mu)$ is a given complete doubling metric measure space supporting a
		$1$-Poincar\'e inequality, then for every $\varepsilon>0$, there exist a function
		$f_\varepsilon$ on $\Omega$ and an open set $U_\varepsilon\subset\Omega$
		such that the following properties hold true:
		\begin{enumerate}
			\item ${\rm Cap}_1(U_\varepsilon)<\varepsilon$;
			\item $\|f-f_\varepsilon\|_{\BV(\Omega)}< \varepsilon$;
			\item  $f^\vee\equiv f_\varepsilon^\vee$ and $f^\wedge\equiv f_\varepsilon^\wedge$ on $\Omega\setminus U_\varepsilon$;
			\item $f_\varepsilon^\vee$ is upper  semicontinuous on $\Omega$, and
			$f_\varepsilon^\wedge$ is lower semicontinuous on $\Omega$.
		\end{enumerate} 
		
		If the space $X$ is unbounded, then such an approximating
		function $f_\varepsilon$ can be constructed  with the additional property that the uniform  limit at infinity of both
		$f^\vee_\varepsilon$ and $f^\wedge_\varepsilon$ is   $0$.
		Moreover, when $X=\R^d$,
		we show that the non-centered maximal function of $f_\varepsilon$ is continuous
		in $\Omega$.
	\end{abstract}
	\tableofcontents
	\section{Introduction}
	
	In this paper, we consider a complete doubling metric measure space $(X,d,\mu)$
	supporting a $p$-Poincar\'e inequality, where $1\leq p<\infty$; see
	Subsection \ref{sec-doublingPoincare} for the relevant definitions.
	For $1\leq p<\infty$, we say that a function
	$f\in L^p(X)$ is a $p$-Haj\l asz-Sobolev function on $(X,d,\mu)$, denoted by
	$f\in M^{1,p}(X)$, if there exists a
	$p$-integrable function $g:X\to [0,\infty]$ such that
	\begin{equation}\label{eq-pw}
		|f(x)-f(y)|\leq d(x,y) \big( g(x)+g(y)\big) \text{\rm \ \ for $\mu$-a.e.} \ x,y\in X.
	\end{equation}
	If $f\in M^{1,p}(X)$, then  the following $M^{1,p}$-norm 
	\[
	\|f\|_{M^{1,p}(X)}:=\|f\|_{L^p(X)}+\inf \|g\|_{L^p(X)}
	\] 
	is finite, where the infimum is taken over all nonnegative functions $g$ for which the pair $(f,g)$ satisfies the pointwise inequality \eqref{eq-pw}.
	Let $Q>1$ be a doubling dimension defined as in \eqref{doubling-dim}. If $p>Q$, then every
	$p$-Haj\l asz-Sobolev function  is locally H\"older continuous on $X$ by the Morrey inequality,
	see e.g.  \cite[Proposition 4.3]{HKi98}. 
	
	For $1\leq p\leq Q$, we say that $f$ is $p$-quasi H\"older continuous on $X$ if for every
	$\varepsilon>0$ and $0<\alpha\leq 1$, there is a set $E_\varepsilon$ with ${\mathcal H}_\infty^{Q-(1-\alpha)p}(E_\varepsilon)<\varepsilon$ so that the {\it restriction} $f|_{X\setminus E_\varepsilon}$ is  $\alpha$-H\"older
	continuous. Here ${\mathcal H}_\infty^{s}$ is the $s$-dimensional Hausdorff
	content.
	In the case $1<p\leq Q$, Haj\l asz and Kinnunen \cite[Theorem 5.3]{HKi98} showed
	that every function $f\in M^{1,p}(X)$ is $p$-quasi H\"older continuous on $X$, and
	furthermore has a
	Lusin approximation as follows. For every $f\in M^{1,p}(X)$, where $1<p\leq Q$, and
	for every $\varepsilon>0$ and $0<\alpha\leq 1$, there are    an open set $U_\varepsilon$
	and a function $f_\varepsilon$ on $X$ so that the following properties hold true:
	\begin{enumerate}
		\item ${\mathcal H}_\infty^{Q-(1-\alpha)p}(U_\varepsilon)<\varepsilon;$
		\item $\| f-f_{\varepsilon\|_{M^{1,p}(X)}}<\varepsilon;$
		\item $f\equiv f_\varepsilon$ on $X\setminus U_\varepsilon$;
		\item $f_\varepsilon$ is $\alpha$-H\"older continuous on $X$.
	\end{enumerate}
	
	In 2007, Kinnunen and  Tuominen \cite[Theorem 4]{KT07} obtained the same result  for
	$1$-Haj\l asz-Sobolev functions $f\in M^{1,1}(X)$.  For higher order or fractional 
	Haj\l asz-Sobolev functions,  analogous results have been obtained in \cite{Ma93, BHS02, Sw02, KP08, HT16}.
	
	Newtonian (Sobolev) functions $f\in N^{1,p}(X)$ are known to be $p$-quasicontinuous,
	where the exceptional set is defined using the $p$-capacity instead of Hausdorff contents.
	We refer to Subsection \ref{sec2.2} for the definitions of Newtonian spaces and capacities.
	Kinnunen, Korte, Shanmugalingam and Tuominen  \cite[Section 7]{KKSN12} proved a
	Lusin approximation for Newtonian spaces $N^{1,p}(X)$, where $1\leq p<\infty$. Namely,
	for every
	$f\in N^{1,p}(X)$ and for every $\varepsilon>0$, there are an open set $U_\varepsilon$
	and a function $f_\varepsilon$ so that the following properties hold:
	\begin{enumerate}
		\item ${\rm Cap}_p(U_\varepsilon)<\varepsilon;$
		\item $\| f-f_{\varepsilon\|_{N^{1,p}(X)}}<\varepsilon;$
		\item $f\equiv f_\varepsilon$ on $X\setminus U_\varepsilon$;
		\item $f_\varepsilon$ is  continuous on $X$.
	\end{enumerate}
	
	A similar result was previously shown in Euclidean spaces in \cite{MZ82} and \cite{Sw07}.
	
	
	For functions of bounded variation, a $1$-quasicontinuity property is too much to hope for.
	However, in 2017  the first author  and Shanmugalingam \cite[Theorem 1.1]{LS17}
	proved that for every
	$f\in \BV(X)$, the upper approximate limit $f^\vee$ and the lower approximate limit
	$f^\wedge$  are respectively $1$-quasi upper and lower semicontinuous on $X$, see
	Subsection \ref{sec2.5} for these definitions.
	On the other hand, the semicontinuity properties may fail in a very large set.
	Consider the function
	\begin{equation}\label{eq1.1-15Nov}
		f=\chi_E {\rm \ \ with\ \ }E:=\bigcup_{j=1}^\infty B(q_j,2^{-j}),
	\end{equation}
	where $q_j$ is a sequence of all points in $\mathbb R^2$ with rational coordinates.
	Here $\chi_E$ denotes the characteristic function of $E$, and  $B(x,r)$ denotes the ball centered at $x$ of radius $r>0$.
	Let $\mathcal L^2$ be the
	$2$-dimensional Lebesgue measure.
	This $f$ is easily seen to be a BV function with
	$\mathcal L^2(E)\leq \pi/2$ and with total variation $\|Df\|(\mathbb R^2)\leq 2\pi$, 
	which is defined in \eqref{eq2.7-2711}.
	Now we have $f^\vee(x)<1$ for all $x\in O_E$, where $O_E$ is
	the measure-theoretic exterior of $E$ defined in \eqref{eq2.7-15Nov},  but then there
	is $q_j\to x$ so that $f^\vee(q_j)\equiv 1$; hence $f^\vee$ is not upper semicontinuous on
	$O_E$, where ${\rm Cap}_1(O_E)\geq \mathcal L^2(O_E)=\infty$.
	Similarly, $(-f)^\wedge$ is not lower semicontinuous on $O_E$. 
	Nonetheless, for a suitable subset $G\subset \mathbb R^2$ with ${\rm Cap}_1(G)$
	smaller than any fixed $\varepsilon>0$, the \emph{restrictions}
	$f^\vee|_{\R^2\setminus G}$ and $f^\wedge|_{\R^2\setminus G}$ are
	respectively upper and lower semicontinuous.
	
	
	In \cite[Page 521]{Pa18},
	an open problem was posed
	on whether it is possible 
	to find a Lusin approximation for a given BV function, satisfying the
	upper/lower semicontinuity properties in the entire space.
	In the following theorem, we give a positive answer to the open problem. 
	
	\begin{mybox}
		\begin{theorem}
			\label{thm1.1}
			Suppose that $(X,d,\mu)$ is a complete doubling metric measure space supporting a
			$1$-Poincar\'e inequality. Let $f\in  {\BV}(\Omega)$, where $\Omega$ is an
			open subset of $X$. Then for every $\varepsilon>0$, there exist  a function
			$f_\varepsilon$ on $\Omega$ and an open set $U_\varepsilon\subset \Omega$
			such that the following properties hold true:
			\begin{enumerate}
				\item ${\rm Cap}_1(U_\varepsilon)<\varepsilon$;
				\item $\|f-f_\varepsilon\|_{\BV(\Omega)}<\varepsilon$;
				\item  $f^\vee\equiv f_\varepsilon^\vee$ and $f^\wedge\equiv f_\varepsilon^\wedge$ on $\Omega\setminus U_\varepsilon$;
				\item $f_\varepsilon^\vee$ is upper  semicontinuous on $\Omega$, and  $f_\varepsilon^\wedge$ is lower semicontinuous on $\Omega$.
			\end{enumerate} 
		\end{theorem}
	\end{mybox}
	A somewhat similar result 
	was given previously 	in Euclidean spaces
	in \cite[Theorem 6.7]{Gro},
	and in the special case of characteristic functions of sets of  finite perimeter in
	\cite[Theorem 5.2]{Pa18}.
	In the case of Newtonian/Sobolev functions, one usually employs a
	discrete convolution $f_\varepsilon$ defined by
	\[
	f_\varepsilon:= \sum_{j\in\mathbb N}\left(\frac{1}{\mu(B_j)}\int_{B_j}f\,d\mu\right) \varphi_j {\rm \ \ on \ }U_\varepsilon,
	\]
	where $\{B_i\}$ is a covering of $U_\varepsilon$ by balls and $\{ \varphi_i\}$ is a partition of unity associated with the covering. However, when $f$ is only a $\BV$ function,
	the semicontinuity property (4) may fail at points on the boundary of $U_\varepsilon$ which
	are approximate jump points of $f$.
	Therefore, in this
	paper, we modify the above definition of  discrete  convolution and use
	approximating Lipschitz functions and
	``boundary truncation" in place of integral averages to obtain the semicontinuity
	property on the boundary of $ U_\varepsilon$.
	
	Recently, the behavior at infinity of BV functions on complete doubling
	unbounded metric measure spaces supporting a $1$-Poincar\'e inequality was studied by the two
	authors in
	\cite{LN24}. Indeed, for every $f\in \BV(X)$, we have the limits
	\begin{equation}\label{eq1.2-26Nov}\notag
		\text{ $\lim_{t\to\infty}f^\vee(\gamma(t))=\lim_{t\to\infty}f^\wedge(\gamma(t))=0$}
	\end{equation}
	for $1$-almost every infinite curve $\gamma$. 
	Here the notion of $1$-a.e. curve is defined in Subsection \ref{sec2.2}, and a locally rectifiable
	curve $\gamma$ is said to be an infinite curve if $\gamma\setminus B\neq\emptyset$ for every
	ball $B$. For the function $f_\varepsilon$ constructed in Theorem \ref{thm1.1},
	we can obtain the additional property that both $f^\vee_\varepsilon$ and $f^\wedge_\varepsilon$
	have a genuine limit $0$ at infinity. We write
	$\Omega\ni x\to\infty$ to mean that $d(O,x)\to\infty$ for a fixed point $O\in X$ and $x\in\Omega$.
	
	\begin{mybox}
		\begin{corollary}\label{cor1.2}
			Suppose that $(X,d,\mu)$ is a complete doubling unbounded metric measure space
			supporting a $1$-Poincar\'e inequality. Let $f\in \BV(\Omega)$ for a given unbounded
			open subset $\Omega$ of $X$. Then for every $\varepsilon>0$, there exists  a
			function $f_\varepsilon\in \BV(\Omega)$ satisfying the properties given in Theorem
			\ref{thm1.1}, and the limits
			\begin{equation}\label{eq1.2-15Nov}
				\lim_{\Omega\ni x\to\infty}f_\varepsilon^\vee(x)=\lim_{\Omega\ni x\to\infty}f_\varepsilon^\wedge(x)=0
			\end{equation}
			hold true.
		\end{corollary}
	\end{mybox}

	In \cite{La21}, the non-centered maximal function  of a BV function $f$,
	denoted by $\mathcal Mf$, was shown to be $1$-quasicontinuous in the Euclidean space
	$\mathbb R^d$, with $d\in\N$.
	As noted in \cite[Example 4.5]{La21}, the non-centered maximal function of the
	BV function $f$ defined as in \eqref{eq1.1-15Nov} on the plane is discontinuous at
	all measure-theoretic exterior points $x$ of $E$ (which constitute a set of infinite $2$-dimensional
	Lebesgue measure), because $\mathcal Mf(q_j)= 1$ and $\mathcal Mf(x)<1$ for all $j$ and for all
	such $x$. In Example \ref{example5.4} we will show that in a complete doubling metric measure space
	supporting a $1$-Poincar\'e inequality, the non-centered maximal function of a BV function
	may fail to be $1$-quasicontinuous.
	Moreover, in this example, we necessarily have $f_\varepsilon\equiv f$ in Theorem \ref{thm1.1}.
	
	Due to all of this, it is natural to ask under what conditions the non-centered maximal function
	of a given BV function $f$ on an
	open set $\Omega\subset\mathbb R^d$ is guaranteed to be continuous.
	The upper semicontinuity of $f_\varepsilon^{\vee}$ given in Theorem \ref{thm1.1},
	while seemingly a rather weak property, is essentially what is required to guarantee continuity
	of the maximal function.
	Indeed, we have the following corollary.
	
	\begin{mybox}
		\begin{corollary}\label{cor1.4}
			Let $\Omega\subset\mathbb R^d$ be an arbitrary nonempty open set, with $d\in\mathbb N$.
			Let $f\in \BV(\Omega)$ be nonnegative. Then for every $\varepsilon>0$, there is a function
			$f_\varepsilon\in \BV(\Omega)$ satisfying the properties given in Theorem \ref{thm1.1},
			and such that $\mathcal M_{\Omega} f_\varepsilon$ is continuous in $\Omega$.
		\end{corollary}
	\end{mybox}

	\section{Preliminaries}\label{sec:prelis}
	
	Throughout this paper, we adopt the following conventions.
	We work in a metric measure space $(X,d,\mu)$, where $\mu$ is a nontrivial locally finite  Borel regular outer measure on $X$.
	We assume that $X$ consists of at least two points.
	The notation $A\lesssim B\ (A\gtrsim B)$ means that there is a constant $C>0$ depending
	only on the space such that $A\leq C \cdot B\ (A\geq C\cdot B)$, and $A\approx B$
	means that $A\lesssim B$ and $A\gtrsim B$. 
	For every $\mu$-measurable subset $A\subset X$ of strictly positive
	and finite measure, and every $f\in L^1(A)$,
	we let $f_A:=\dashint_Af\,d\mu=\frac{1}{\mu(A)}\int_Af\,d\mu$.
	We denote an open ball centered at $x\in X$ of radius $r>0$ by $B(x,r)$,
	and for $\tau>0$, we set $\tau B(x,r):=B(x,\tau r)$. 
	Given an open set $\Omega\subset X$, we say that $f\in L^1_{\loc}(\Omega)$ if
	for every $x\in \Omega$ there exists $r>0$ such that $f\in L^1(B(x,r))$.
	Other local function spaces are defined similarly.
	
	\subsection{Modulus, capacity and Newtonian spaces}\label{sec2.2}\
	
	A \emph{curve} is a nonconstant continuous mapping from an interval $I\subset \R$ into $X$.
	A curve is \emph{rectifiable} if its length is finite, and \emph{locally rectifiable}
	if its restriction to every compact subinterval of $I$ is rectifiable.
	We assume every rectifiable curve $\gamma$ to be parametrized by arc length
	as $\gamma: [0,\ell_{\gamma}]\to X$, and then the curve integral of a nonnegative Borel
	function $\rho$ is given by
	\[
	\int_{\gamma}\rho\,ds:=\int_0^{\ell_{\gamma}}\rho(\gamma(t))\,dt.
	\]
	If $\gamma$ is locally rectifiable, then we set
	\[
	\int_{\gamma}\rho\,ds:=\sup \int_{\gamma'}\rho\,ds,
	\]
	where the supremum is taken over all rectifiable subcurves $\gamma'$ of $\gamma$.
	
	Let $\Gamma$ be a family of curves in $X$. Let $1\le p<\infty$.
	The  \emph{$p$-modulus} of $\Gamma$ is defined by 
	\[
	\text{\rm Mod}_p(\Gamma):=\inf \int_{X}\rho^p\,d\mu,
	\]
	where the infimum is taken over all Borel functions $\rho:X\to[0,{\infty}]$ satisfying 
	$\int_{\gamma}\rho \,ds\geq 1$ for every locally rectifiable curve $\gamma\in\Gamma$.
	A family of curves is called \textit{$p$-exceptional} if it has $p$-modulus zero. We say
	that a property  holds for \textit{$p$-a.e. curve} (or $p$-almost every curve) if the
	collection of curves for which the property fails  is $p$-exceptional. 
	
	Let $\Omega\subset X$ be open and let $f:\Omega\to [-\infty,\infty]$ be measurable.
	A Borel function $\rho:\Omega\to[0,{\infty}]$ is said to be an \textit{upper gradient} of $f$ in 
	$\Omega$ if 
	\begin{equation}
		\label{def-upper-gradient}|f(x)-f(y)|\leq \int_{\gamma}\rho\, ds
	\end{equation}
	for every rectifiable curve $\gamma$ in $\Omega$ connecting $x$ and $y$. 
	For $1\le p<\infty$, we say that a $\mu$-measurable function $\rho:\Omega\to[0,{\infty}]$
	is a  \textit{$p$-weak upper gradient}
	of $f$ if \eqref{def-upper-gradient} holds for $p$-a.e. rectifiable curve.
	If $f$ has a $p$-weak upper gradient in $\Omega$
	that is in $L_{\loc}^p(\Omega)$, we denote by $g_f$ the
	\textit{minimal $p$-weak upper gradient} of $f$ in $\Omega$,
	which exists and is unique up to sets of measure zero,
	and is minimal in the sense that $g_f\leq\rho $ a.e.\ for every
	$\rho\in L_{\loc}^p(\Omega)$ that is a $p$-weak upper gradient of $f$ in $\Omega$.
	In \cite{H03}, the existence and uniqueness of such a minimal
	$p$-weak upper gradient are given.  The notion of upper gradients is due to Heinonen and Koskela
	\cite{HK98}, and we refer  interested  readers to \cite{BB11,H03,HK98,HKST15,N00} for a more
	detailed discussion. 	
	
	For $1\leq p<\infty$,  the Newtonian space, denoted by $N^{1,p}(\Omega)$, is the collection of
	all $p$-integrable functions on $\Omega$ whose minimal $p$-weak upper gradient is
	also in $L^p(\Omega)$.
	The Newtonian norm 
	is defined by
	\[
	\|f\|_{N^{1,p}(\Omega)}:=\|f\|_{L^p(\Omega)} + \|g_f\|_{L^p(\Omega)}.
	\]
	
	The \textit{$p$-capacity} of a set $K\subset X$ 
	is defined by 
	\[\text{\rm Cap}_p(K):=\inf \left(\int_{X}|u|^p\,d\mu+\int_{X}g_u^p\,d\mu\right),
	\] 
	where the infimum is taken over all functions $u\in N^{1,p}(X)$ such that $u|_K\equiv 1$.
	
	\subsection{Doubling measures and Poincar\'e inequalities}\label{sec-doublingPoincare}\
	
	The Borel regular outer measure $\mu$ is said to be \textit{doubling} if
	there exists a finite constant $C_d\geq 1$ such that for all balls $B(x,r)$
	with radius $r>0$ and center $x\in X$, we have
	\begin{equation}\label{eq2.4-0912}
		0<\mu(B(x,2r))\leq C_d\mu(B(x,r))<\infty.
	\end{equation}
	Here $C_d$ is called the \textit{doubling constant}.
	
	By iterating the doubling condition \eqref{eq2.4-0912}, we obtain a relative lower volume
	decay of order $0<Q<\infty$: there is a constant $0<C_Q<\infty$ such that
	\begin{equation}\label{doubling-dim}
		\left( \frac{r}{R}\right)^Q\leq C_Q \frac{\mu(B(y,r))}{\mu(B(x,R))}
	\end{equation}
	for all $0<r\leq R<\infty$ and $y\in B(x,R)$. The choice $Q=\log_2(C_d)$ works, but a smaller value of $Q$ might satisfy the above condition as well. We may assume that $Q>1$.
	
	Let $1\leq p<{\infty}$. We say that $X$ supports a  \textit{$p$-Poincar\'e inequality}
	if there exist finite constants $C_P>0$ and $\lambda\geq 1$ such that 
	\begin{equation}\label{eq2.5-2711}
		\dashint_{B(x,r)}|f-f_{B(x,r)}|\,d\mu 
		\leq C_P r\left (\dashint_{B(x,\lambda r)}\rho^p\,d\mu\right )^{1/p}
	\end{equation}
	for all balls $B(x,r)$ with radius $r>0$ and center $x\in X$, and for all pairs
	$(f,\rho)$ satisfying \eqref{def-upper-gradient} such that $f$ is integrable on balls.
	Here $\lambda$ is called the {\it scaling constant} or {\it scaling factor}
	of the Poincar\'e inequality,
	and $C_P$ is called the constant of the Poincar\'e inequality.
	For more on Poincar\'e inequalities, we refer the interested reader to \cite{HK95,HK00,Hei01,HKST15}.
	
	We will assume throughout the paper that $X$ is complete, $\mu$ is doubling, and that $X$ supports 
	a $1$-Poincar\'e inequality.
	It is customary to express this by saying that
	$(X,d,\mu)$ is a complete doubling metric measure space supporting a $1$-Poincar\'e inequality.
	
	From \cite[Lemma 3.5]{Pa18}, we have the following lemma.
	\begin{mybox}
		\begin{lemma}\label{lem2.2-19Jul}
			Let $\varepsilon>0$ and let $G\subset X$. 
			Then there exist a finite constant $C_1>0$
			and an open set $U\supset G$ such that 
			${\rm Cap}_1(U)\leq C_1 {\rm Cap}_1(G)+\varepsilon$ and
			\[
			\frac{\mu(B(x,r)\cap G)}{\mu(B(x,r))}\to 0 
			\]
			as $r\to0$ uniformly for all $x\in X\setminus U$.
		\end{lemma}
	\end{mybox}
	\subsection{Functions of bounded variation}\label{sec2.5}\
	
	In this subsection, we recall the definition and basic properties of functions of bounded variation on metric spaces as in \cite{Mi03}.
	
	Let $\Omega\subset X$ be open and let $f\in L^1_{\rm loc}(\Omega)$. We define the total variation
	of $f$ in $\Omega$ by
	\begin{equation}\label{eq2.7-2711}
		\| Df\|(\Omega):=\inf\left\{ \liminf_{i\to\infty}\int_\Omega g_{f_i}\,d\mu: f_i\in {\rm Lip}_{\rm loc}(\Omega), f_i\to f {\rm\ in\ }L^1_{\rm loc}(\Omega)\right\}
	\end{equation}
	where each $g_{f_i}$ is the minimal $1$-weak upper gradient of $f_i$ in $\Omega$. Here ${\rm Lip}_{\rm loc}(\Omega)$ is the  collection of locally Lipschitz functions on $\Omega$.
	We say that an integrable function $f\in L^1(\Omega)$ is of bounded variation, denoted
	$f\in \BV(\Omega)$, if the total variation  $\|Df\|(\Omega)$ is finite.
	For an arbitrary subset $A\subset X$, we define 
	\[
		\| Df\|(A):=\inf \{\| Df\|(W): A\subset W, \,W\subset X {\rm \ is\ open} \}.
	\]
	By \cite[Theorem 3.4]{Mi03}, $\|Df\|$ is a Radon measure on a given open set $\Omega$ for
	$f\in L^1_{\rm loc}(\Omega)$ with $\|Df\|(\Omega)<\infty$. 
	The $\BV$ norm is defined by
	\[
	\|f\|_{\BV(\Omega)}:=\|f\|_{L^1(\Omega)}+ \|Df\|(\Omega).
	\]
	
	
	From \cite[Lemma 2.2]{KKSL13}, we have the following version of the Poincar\'e inequality.
	\begin{mybox}
		\begin{lemma}\label{lem2.2-Sep}
			Suppose that $(X,d,\mu)$ is a complete doubling metric measure space supporting a $1$-Poincar\'e inequality. Let $f\in \BV(X)$ and let $B$ be a ball so that $\mu(\{x\in B: |f(x)|>0 \})<  c_1 \mu(B)$ for some $0<c_1<1$. Then 
			\[
			\dashint_{B} |f|\,d\mu\leq C_2 \frac{r }{1-c_1^{1/Q}} \frac{\|Df\|(\lambda B)}{\mu(B)},
			\]
			where the constant $C_2>0$ only depends on the constants of the doubling condition and 
			the Poincar\'e inequality. Here $\lambda\geq 1$ is the scaling constant of the
			Poincar\'e inequality, and $Q>1$ is the order of the relative lower volume
			decay \eqref{doubling-dim}.
		\end{lemma}
	\end{mybox}

	A $\mu$-measurable set $E\subset X$ is said to be of finite perimeter if $\|D\chi_E\|(X)<\infty$,
	where $\chi_E$ is the characteristic function of $E$. The perimeter of $E$
	in $A\subset X$ is denoted by 
	\[
	P(E,A):=\|D\chi_E\|(A).
	\]
	
	We have the following coarea formula from \cite[Proposition 4.2]{Mi03}:
	if $f\in \BV(\Omega)$ for $\Omega\subset X$ open and $F\subset \Omega$ is a Borel set, then
	\begin{equation}\label{eq2.8-Oct}
		\|Df\|(F)=\int_{-\infty}^{\infty}P(\{f>t\}, F)\,dt.
	\end{equation}

	The measure-theoretic interior $I_E$ and exterior $O_E$ of a set $E\subset X$ are defined respectively by
	\[
	I_E:=\left\{x\in X:  \lim_{r\to 0} \frac{\mu(B(x,r)\setminus E)}{\mu(B(x,r))}=0 \right\}
	\]
	and
	\begin{equation}\label{eq2.7-15Nov}
		O_E:=\left\{x\in X: \lim_{r\to 0} \frac{\mu(B(x,r)\cap E)}{\mu(B(x,r))}=0 \right\}.
	\end{equation}
	The measure-theoretic boundary $\partial^*E$ is defined as the set of points $x\in X$ at which both $E$ and its complement have strictly positive upper density, i.e.
	\[
	\limsup_{r\to 0} \frac{\mu(B(x,r)\setminus E)}{\mu(B(x,r))}>0\ \ {\rm and\ \ }  \limsup_{r\to 0} \frac{\mu(B(x,r)\cap E)}{\mu(B(x,r))}>0.
	\]
	Then $X=I_E\cup O_E\cup\partial^*E$.
	
	We define the  \emph{codimension $1$ Hausdorff content} of a set $A\subset X$, for  $0<R\le \infty$,
	as follows:
	\[
	{\mathcal H}_R(A):=\inf\left\{\sum_{k\in\mathbb N}\frac{\mu(B_k)}{r_k}:A\subset\bigcup_{k\in\mathbb N}B_k\ \ \text{and}\ \ 0<r_k<R\right\}.
	\]
	The \emph{codimension $1$ Hausdorff measure} is then defined by 
	\begin{equation}\label{2.7-8thMarch}\notag
		{\mathcal H}(A):=\lim_{R\to 0}{\mathcal H}_R(A)=\sup_{R>0} {\mathcal H}_R(A).
	\end{equation}
	

	Given an open set $\Omega\subset X$ and a $\mu$-measurable subset $E\subset X$ with $P(E,\Omega)<\infty$, we have that for any Borel set $A\subset \Omega$,
	\begin{equation} \label{eq2.5-0411}
		P(E,A)=\int_{\partial^*E\cap A}\theta_E\,d\mathcal H
	\end{equation}
	with $\theta_E:X\to[\alpha,C_d]$, where $C_d$ is the doubling constant and
	$\alpha>0$ depends only on
	the doubling constant and constants of the $1$-Poincar\'e inequality,
	see \cite[Theorem 5.3 and Theorem 5.4]{Am02}
	and \cite[Therem 4.6]{AMP04}.
	
	The lower and upper approximate limits of a function $f$  on $X$ are defined respectively by 
	\begin{equation}\label{eq2.6-1511}\notag
		f^\wedge(x):= \sup\left\{ t\in\mathbb R: \lim_{r\to 0}\frac{\mu(\{y\in B(x,r): f(y)<t\}) }{\mu(B(x,r))}=0 \right\}
	\end{equation}
	and
	\begin{equation}\label{equ2.12-0912}\notag
		f^\vee(x):= \inf\left\{ t\in\mathbb R: \lim_{r\to 0}\frac{\mu( \{y\in B(x,r): f(y)>t\}) }{\mu(B(x,r))}=0 \right\}
	\end{equation}
	for $x\in X$.
	When studying the fine properties of functions of bounded variation, we consider the pointwise representatives $f^\vee$ and $f^\wedge$. 
	
		
		
		Given an open set $\Omega\subset X$, we say that a function $f$ is $1$-quasi
		(lower/upper semi-) continuous on $\Omega$ if for every $\varepsilon>0$ there exists
		an open set $G\subset X$ such that ${\rm Cap}_1(G)<\varepsilon$ and
		$f|_{\Omega\setminus G}$ is finite and (lower/upper semi-) continuous.
		
		By Corollary 4.2 in \cite{Pa18}, or Theorem 1.1 in \cite{LS17}, we obtain that
		BV functions have a $1$-quasi semicontinuity property as below.
		\begin{mybox}
			\begin{proposition}\label{prop2.2-28Sep} 
				Let  $\Omega\subset X$ be open and let $f\in \BV(\Omega)$.
				Then $f^\wedge$ is $1$-quasi lower semicontinuous on $\Omega$
				and $f^\vee$ is $1$-quasi upper semicontinuous on $\Omega$.
			\end{proposition}
		\end{mybox}
		\subsection{Whitney decomposition} \label{sec2.4}\
		
		Recall that we assume $(X,d,\mu)$ to be a doubling metric measure space supporting a $1$-Poincar\'e inequality, where
		$\lambda$ is  the scaling  factor of the $1$-Poincar\'e inequality \eqref{eq2.5-2711}.
		Given $x\in X$ and $A\subset X$, we denote
		\[
		\dist(x,A):=\inf\{d(x,y)\colon y\in A\}.
		\]
		Let $U\subset X$ be open. We can choose a Whitney-type covering 
		of $U$ as follows. See e.g. \cite[Proposition 4.1.15]{HKST15} for a similar construction; this is based on
		earlier works, e.g. \cite{CW71}.
		We have a Whitney-type covering $\{ B_j:=B(x_j,r_j)\}_{j\in\mathbb N}$ of $U$ such that 
		\begin{enumerate}
			\item for each $j\in\mathbb N,$ 
			\[
			r_j=\min\left\{1,\frac{{\rm dist}(x_j,X\setminus U)}{8\lambda}\right\};
			\]
			\item for each $k\in\mathbb N$, the ball $2\lambda B_k$ meets at most $C_0:=C_0(C_d,\lambda)$ balls $2\lambda B_j$ (that is, a bounded overlap property holds);
			\item if $2\lambda B_j$ meets $2\lambda B_k$, then $r_j\leq 2r_k$.
		\end{enumerate}
		Given such a covering of $U$, we can take a partition of unity $\{ \varphi_j\}_{j\in\mathbb N}$ subordinate to the
		covering, such that $0\leq \varphi_j\leq 1$, each $\varphi_j$ is a $C_W/r_j$-Lipschitz function for some constant
		$C_W\geq 1$, and $\supp \varphi_j \subset 2B_j$ for each $j\in\mathbb N$, see e.g.
		\cite[p. 104]{HKST15} or \cite{BBS07}. 
		
		
		\subsection{Functions of bounded variation in $\mathbb R^d$}\ 
		
		In this subsection, we consider the Euclidean theory that will be needed
		(only) in Section \ref{sec:Euclidean}. Consider the  $d$-dimensional Euclidean space
		$\mathbb R^d=(\mathbb R^d,|\cdot|,\mathcal L^d)$ equipped with the Euclidean distance $|\cdot|$ and the $d$-dimensional Lebesgue measure $\mathcal L^d$, with $d\in\mathbb N$.
		We follow the monograph \cite{AFP00}.
		We always let $\Omega\subset \mathbb R^d$ be an open set.
		We let the class $\BV(\Omega)$ be defined as in Subsection \ref{sec2.5}.
		
		Let $f\in L_{\loc}^1 (\Omega)$. 
		For $x\in \Omega$, we say that $\widetilde{f}(x)\in\mathbb R$ is a Lebesgue
		representative of $f$ at $x$ if
		\begin{equation}\label{def-Leb}
			\lim_{r\to0}\,\dashint_{B(x,r)}|f(y)-\widetilde{f}(x)|\,dy=0.
		\end{equation}
		We denote by $S_f\subset \Omega$ the set where such a representative does not exist
		and call it the approximate
		discontinuity set or the set of non-Lebesgue points.
		Given $x\in \Omega$, $r>0$, and a point on the unit sphere $\nu\in\mathbb S^{d-1}$, we define the half-balls
		\[
		B_\nu^+(x,r):=\{ y\in B(x,r): \langle y-x, \nu\rangle>0\} {\rm \ \ and\ \ }B_\nu^-(x,r):=\{ y\in B(x,r): \langle y-x,\nu\rangle<0\}.
		\]
		Here $\langle \cdot, \cdot\rangle$ denotes the Euclidean inner product.
		We say that $x\in \Omega$ is an approximate jump point of $f$ if there is $\nu\in\mathbb S^{d-1}$ and distinct numbers $f^+(x), f^-(x)\in\mathbb R$ such that 
		\begin{equation}\label{def-f+}
			\lim_{r\to0}\,\dashint_{B_\nu^+(x,r)}|f(y)-f^+(x)|\,dy=0
		\end{equation}
		and
		\begin{equation}\label{def-f-}
			\lim_{r\to0}\,\dashint_{B_\nu^-(x,r)}|f(y)-f^-(x)|\,dy=0.
		\end{equation}
		The set of all approximate jump points is denoted by $J_f\subset\Omega$. By \cite[Theorem 3.78]{AFP00},
		we know that $\mathcal H^{d-1}(S_f\setminus  J_f)=0$, where $\mathcal H^{d-1}$ is
		the $(d-1)$-dimensional Hausdorff measure on $\mathbb R^d$. By \cite[Theorem 4.3 and Theorem 5.1]{HK10},
		it follows that
		\begin{equation}\label{Sf minus Jf}
			{\rm Cap}_1(S_f\setminus J_f)=0.
		\end{equation}
		Note that for all $x\in \Omega\setminus S_f$,
		we have $\widetilde{f}(x)=f^\vee(x)=f^\wedge(x)$ where $f^\vee$ and $f^\wedge$
		are respectively the upper and
		lower approximate limits of $f$ defined as in Subsection \ref{sec2.5}.
		Furthermore, for all $x\in J_f$,
		we have that $f^\vee(x)=\max\{ f^-(x), f^+(x)\}$ and $f^\wedge(x)= \min\{ f^-(x), f^+(x)\}$.
		In total,
		\begin{equation}\label{eq4.6-15Oct}
			\begin{cases}
				f^\vee(x)=\max\{f^-(x), f^+(x), \widetilde{f}(x) \}\\
				f^\wedge(x)=\min\{f^-(x), f^+(x), \widetilde{f}(x) \} 
			\end{cases}
			{\rm \ \ for\ all\ }x\in (\Omega\setminus S_f) \cup J_f.
		\end{equation}
		(The number $\widetilde{f}(x)$ is only defined for $x\in \Omega\setminus S_f$ and
		the numbers $f^-(x), f^+(x)$ are only defined for $x\in J_f$, but the above maximum and minimum
		have the obvious interpretations.)
		The non-centered maximal function of a function $f\in L^1_{\loc}(\Omega)$ is defined by 
		\[
		\mathcal M_{\Omega}f(x):=\sup_{x\in B(z,r)\subset \Omega}\,\dashint_{B(z,r)}|f| \,dy,\quad x\in\Omega.
		\]
		If $\Omega=X$, we omit it.
		By \cite[Lemma 3.1]{La21}, we have that for $x\in \Omega$, 
		\begin{equation} \label{equ4.6-13Oct}
			\mathcal M_{\Omega} f(x)=\sup_{x\in \overline{B}(z,r),\, B(z,r)\subset \Omega}\,\dashint_{B(z,r)}|f|\,dy.
		\end{equation}
		Here $\overline{B}(z,r)$ denotes the closed ball.
		
		\section{Proof of Theorem \ref{thm1.1}} \label{sec3}
		
		This section is devoted to proving our main Theorem \ref{thm1.1}.
		Recall that we assume that $(X,d,\mu)$ is a complete doubling metric measure space  supporting 
		a $1$-Poincar\'e inequality.
		
		We recall also that $C_d\geq1, \lambda\geq 1$ are, respectively, the doubling constant and the
		scaling factor of the $1$-Poincar\'e inequality \eqref{eq2.5-2711}. The constant of the relative
		lower volume decay \eqref{doubling-dim}  is $C_Q>0$, with $Q>1$.
		
		Let $\varepsilon>0$ be arbitrary; without loss of generality we can assume that
		$\varepsilon \le 1$. We consider $f\in {\BV}(\Omega)$, where $\Omega\subset X$ is open.
		By Proposition \ref{prop2.2-28Sep},
		there is
		an open subset $G$ of $\Omega$ with ${\rm Cap}_1(G)<\varepsilon/C_1$, where 
		$C_1$ is the constant from Lemma \ref{lem2.2-19Jul}, and
		so that $f^\vee|_{\Omega\setminus G}$ and $f^\wedge|_{\Omega\setminus G}$ are respectively upper and lower semicontinuous.
		
		By Lemma \ref{lem2.2-19Jul}, we find $0<R<1$ and an open set $U'$ with
		$G\subset U'\subset X$ and
		\begin{equation}\notag
			{\rm Cap}_1(U') < \varepsilon
			\ \
			{\rm    and   }
			\ \     \frac{\mu(B(x,r)\cap G)}{\mu(B(x,r))}<\min\left\{1,\frac{1}{(20 \lambda)^QC_Q}\right\}
		\end{equation}
		for  all $x\in X\setminus U'$ and $0<r<R<1$.
		Having chosen $\varepsilon$ smaller if necessary,
		which we can do without loss of generality, we can also assume that $X\setminus U'$ is nonempty.
		
		The variation measure is absolutely continuous with respect to the $1$-capacity in the following sense.
		\begin{mybox}
			\begin{lemma}[{\cite[Lemma 3.8]{Pa18}}]\label{lem:variation measure and capacity}
				Let $\Omega\subset X$ be an open set and
				let $f\in L^1_{\rm loc}(\Omega)$ with $\Vert Df\Vert(\Omega)<\infty$. Then for every
				$\varepsilon'>0$ there exists $\delta>0$ such that if $A\subset \Omega$ with ${\rm Cap}_1 (A)<\delta$,
				then $\Vert Df\Vert(A)<\varepsilon'$.
			\end{lemma}
		\end{mybox}
		By this lemma, we may also assume that
		$\|Df\|(U'\cap \Omega)<\varepsilon$. Then let $U:=U'\cap \Omega$.
		We gather all of the known properties:
		\begin{equation}\label{equ3.1-28Sep}
			{\rm Cap}_1(U)<\varepsilon, \ \ \|Df\|(U)<\varepsilon, \ \ 
			\frac{\mu(B(x,r)\cap G)}{\mu(B(x,r))}<\min\left\{1,\frac{1}{(20\lambda)^Q C_Q}\right\}
		\end{equation}
		for  all $x\in \Omega\setminus U$ and $0<r<R<1$,
		and
		\begin{equation}\label{equ3.2-28Sep}
			f^\vee|_{\Omega\setminus G}\ \textrm{ and }\ f^\wedge|_{\Omega\setminus G}
			\quad \textrm{are respectively upper and lower semicontinuous.}
		\end{equation}
		Let $\{ B_j:=B(x_j,r_j)\}_{j\in\mathbb N}$ and $\{\varphi_j\}_{j\in\mathbb N}$ be the Whitney balls  and the partition of unity defined as in Subsection \ref{sec2.4}. 
		Let
		\begin{equation}\label{equa3.6-Sep}
			U_1:= \left\{ x\in U: \dist(x, \Omega\setminus U)< 
			\frac{\min\{ R, {\rm dist}(x, X\setminus \Omega)\}}{100\lambda}\right\}
		\end{equation}
		and $U_2:=U\setminus U_1$.
		Let
		$J_1:= \{j\in\mathbb N: x_j\in U_1\} $
		and $J_2:=\mathbb N\setminus J_1$.
		\begin{mybox}
			\begin{lemma}\label{lem3.1-28Sep}
				For all $j\in J_1$, we have
				\[
					\mu(B_j\cap G)
					\leq \frac 12 \mu(B_j).
				\]
			\end{lemma}
		\end{mybox}
		\begin{proof}
			Let $j\in J_1$.
			We find a point $x_j'\in X\setminus U$ such that
			$\dist(x_j,X\setminus U)=d(x_j,x_j')$.
			Since $x_j\in U_1$, we necessarily have $x_j'\in \Omega\setminus U$.
			Since $d(x_j, X\setminus U)= 8\lambda r_j$, we have
			\[
			B_j\subset  B(x_j', 9\lambda r_j).
			\]
			It follows from  the property of the relative lower volume decay and the estimate \eqref{equ3.1-28Sep}
			that
			\begin{align*}
				\mu(B_j\cap G)
				&\leq  \mu(B(x_j', 9\lambda  r_j)\cap G)\\
				&\overset{\eqref{equ3.1-28Sep}}{\leq}  \frac{1}{(20\lambda)^Q C_Q}\mu(B(x_j', 9\lambda r_j))\\
				&\overset{\eqref{doubling-dim}}{\leq}  \frac{1}{(20\lambda)^Q C_Q} C_Q \left( \frac{9 \lambda  r_j} {r_j}\right)^Q \mu(B_j)\\
				&\leq \frac 12 \mu(B_j).
			\end{align*}
			The proof is completed.
		\end{proof}
		
		We now construct the function $f_\varepsilon$ and show that
		it satisfies the properties given in the theorem.
		Note that for all $j\in\N$,
		we have $2\lambda B_j\Subset U\subset \Omega$ and then
		$f\in L^1(2\lambda B_j)$.
		Moreover, in the definition of the total variation $\|Df\|(2\lambda B_j)$, one can consider
		approximating locally Lipschitz function converging in $L^1(2\lambda B_j)$ and
		not only in $L_{\loc}^1(2\lambda B_j)$, see e.g. \cite[Lemma 3.2]{La20}.
		Applying also Egorov's theorem, for each $j\in\mathbb N$ we find a function 
		$f_j\in \liploc(2\lambda B_j)$
		and a set
		$N_j\subset 2B_j$ such that
		\begin{equation}
			\begin{cases}
				\|Df_{j}\|(2\lambda B_j)
				\leq \|Df\|(2\lambda B_j) +  2^{-j}\varepsilon,\\
				\| f_j-f\|_{L^1({2\lambda B_j})}\leq 2^{-j}\varepsilon r_j,\\ 
				|f_j(y)-f(y)|\leq 2^{-j}  {\rm \ \ for \ all\ }y\in 2 B_j\setminus N_j,\\
				\mu(N_j)\leq 
				2^{-j} \min\{1, \mu(B_j)/100\}  .
			\end{cases}\label{eq3.7-Sep}
		\end{equation}
		
		For each $j\in\mathbb N$, we let 
		\begin{equation}\label{def-aj}
			a_j(x):= 
			\begin{cases}
				\max\{ m_j, \min\{M_j, f_j(x)\}\}, & \ \ {\rm if\ }j\in J_1,\\
				f_j(x), &{\rm \ \ if\ }j\in J_2,
			\end{cases}
		\end{equation}
		for $x\in 2B_j$, where $m_j, M_j$ are defined by 
		\begin{equation}\label{def-Mjmj}
			m_j:= {\rm essinf}_{x\in 2B_j\setminus (G\cup N_j)}f_j(x) 
			\ \ 
			{\rm and}
			\ \
			M_j:= {\rm esssup}_{x\in 2B_j\setminus (G\cup N_j)}f_j(x).
		\end{equation}
		Note that $a_j$ is continuous on $2B_j$ for each $j\in\mathbb N$.
		We have from Lemma \ref{lem3.1-28Sep} and the fourth property of \eqref{eq3.7-Sep}
		that for $j\in J_1$,
		\begin{equation}\label{eq3.8-Sep}
			\mu(2B_j\setminus (G\cup N_j))
			\geq   \left( 1- \frac{1}{2} - \frac{1}{2^j}\frac{1}{100}\right) \mu(B_j)
			\ge \frac{\mu(B_j)}{4}
			>0,
		\end{equation}
		and so $M_j, m_j$ are well defined. We define the function $f_\varepsilon$ on $\Omega$ by setting
		\begin{equation}\label{def-f}
			f_\varepsilon:= f \chi_{\Omega\setminus U} + \sum_{j\in\mathbb N} a_j \varphi_j,
		\end{equation}
		where the functions $\varphi_j$ are the partition of unity with respect to the
		Whitney covering $\{ B_j\}_{j\in\mathbb N}$.
		\begin{mybox}
			\begin{lemma}
				\label{lem3.4-Sep} For all $x\in \Omega\setminus U$,  we have $f_\varepsilon^\vee(x)=f^\vee(x)$ and $f_\varepsilon^\wedge(x)=f^\wedge(x)$.
			\end{lemma}
		\end{mybox}
		\begin{proof}
			Since $f_\varepsilon\equiv f$ on $\Omega\setminus U$ and $\Omega$ is open, by the definition of upper and lower approximate limits we have that $f^\vee(x)=f^\vee_\varepsilon(x)$ and $f^\wedge(x)=f^\wedge_\varepsilon(x)$ for all $x\in \Omega\setminus \overline{U}$.
			
			It remains to consider the case $ x\in \partial U\cap\Omega$. Recall from \eqref{equ3.1-28Sep} that
			\begin{equation}\label{eq3.13-Oct}
				\lim_{r\to0}\frac{\mu(B(x,r)\cap G)}{\mu(B(x,r))}=0.
			\end{equation}
			If a dilated Whitney ball $2B_j$ intersects $B(x,r)$, then clearly
			$2B_j\subset B(x,2r)$.
			Then the fourth property of \eqref{eq3.7-Sep} gives
			\begin{align}
				\limsup_{r\to0}\sum_{j\in J_1:\, 2B_j\cap B(x,r)\neq\emptyset}\frac{\mu(N_j)}{\mu(B(x,r))}
				& \leq  \limsup_{r\to 0} \sum_{j\in J_1:\, 2B_j\cap B(x,r)\neq\emptyset}2^{-j} \frac{\mu(B_j)}{\mu(B(x,r))}\notag\\
				& \leq  \frac{\mu(B(x,2r))}{\mu(B(x,r))} \limsup_{r\to 0} \sum_{j\in J_1:\, 2B_j\cap B(x,r)\neq\emptyset}2^{-j}=0,\label{eq3.14-Sep}
			\end{align}
			since $2B_j\cap B(x,r)=\emptyset$ for any fixed $j$ and sufficiently small $r$.
			Let $r>0$ be so that $B(x,2r)\cap U\subset U_1$.
			By the third property of \eqref{eq3.7-Sep}, we have that for all $y\in B(x,r)\setminus\left(G\cup \cup_{j\in J_1}N_j \right)$,
			\begin{align*}
				|f_\varepsilon(y)-f(y)|
				&\leq \sum_{j\in J_1:\, 2B_j\cap B(x,r)\neq\emptyset} |a_j(y)-f(y)|\chi_{2B_j}(y)\\
				&= \sum_{j\in J_1:\, 2B_j\cap B(x,r)\neq\emptyset} |f_j(y)-f(y)|\chi_{2B_j\setminus (G\cup N_j)}(y)\\
				&\leq  \sum_{j\in J_1:\, 2B_j\cap B(x,r)\neq\emptyset} 2^{-j}.
			\end{align*}
			Notice that the last line of the above estimate converges to $0$ as $r\to0$. Let $\delta>0$ be arbitrary; now there is $r_\delta>0$ such that 
			\[
			\label{eq3.15-Sep} |f_\varepsilon(y)-f(y)|< \delta
			\]
			for all $y\in B(x,r)\setminus \left( G\cup\cup_{j\in J_1}N_j\right)\subset\Omega$ with $0<r<r_\delta$. 
			Using this fact and then \eqref{eq3.13-Oct} and \eqref{eq3.14-Sep}, we get
			\begin{equation}\label{eq:f epsilon f}
				\begin{split}
					&\limsup_{r\to 0}\frac{\mu(\{y\in B(x,r): |f_\varepsilon(y)-f(y)|>\delta \})}{\mu(B(x,r))}\\
					&\quad  \le \limsup_{r\to 0}\frac{\mu(B(x,r)\cap G)}{\mu(B(x,r))}
					+\limsup_{r\to 0} \sum_{j\in J_1: 2B_j\cap B(x,r)\neq\emptyset} \frac{\mu(N_j)}{\mu(B(x,r))} \\
				&\quad= 0.
			\end{split}
		\end{equation}
		Thus we get
		\begin{equation}\label{eq3.13-3Dec}
			\begin{split}
				&\limsup_{r\to 0}\frac{\mu(\{y\in B(x,r): f_\varepsilon(y)>f^\vee(x)+2\delta \})}{\mu(B(x,r))}\\
				&\quad \leq \limsup_{r\to 0}\frac{\mu(\{y\in B(x,r): |f_\varepsilon(y)-f(y)|+f(y)>f^\vee(x)+2\delta \})}{\mu(B(x,r))}\\
				&\quad\leq \limsup_{r\to 0}\frac{\mu(\{y\in B(x,r):  f(y)>f^\vee(x)+\delta \})}{\mu(B(x,r))}\\
				&\quad=0.
			\end{split}
		\end{equation}
		As $\delta>0$ was arbitrary, this implies that $f_\varepsilon^\vee(x)\leq f^\vee(x)$ by
		the definition of the upper approximate limit $f^\vee_\varepsilon(x)$. By symmetry of $f$
		and $f_\varepsilon$ from \eqref{eq:f epsilon f}, we also get
		$f^\vee(x)\leq f_\varepsilon^\vee(x)$ and so $f^\vee(x)=f^\vee_\varepsilon(x)$.
		Similarly, one obtains that $f^\wedge(x)=f^\wedge_\varepsilon(x)$. Then we conclude
		that for all $x\in \Omega\setminus U$, $f_\varepsilon^\vee(x)=f^\vee(x)$
		and $f_\varepsilon^\wedge(x)=f^\wedge(x)$. 
		
		The proof is completed.
	\end{proof}
	
	We denote by $f_+:=\max\{0, f \}$ the positive
	part of the function $f$.
	\begin{mybox}
		\begin{lemma}\label{lem3.2-28Sep}
			We have $\|D(f_\varepsilon-f)\|(\Omega)\lesssim \varepsilon$.
		\end{lemma}
	\end{mybox}
	\begin{proof}
		
		For any $v,w\in L^1_{\loc}(\Omega)$, it is straightforward to show from the definition of 
		the total variation that
		\begin{equation}\label{eq:BV functions form vector space}
			\Vert D(u+v)\Vert(\Omega)\le \Vert Du\Vert(\Omega)+\Vert Dv\Vert(\Omega).
		\end{equation}
		By the first property of \eqref{eq3.7-Sep},
		we have
		\begin{equation}\label{eq:aj minus f estimate}
			\|D(a_j-f)\|(2B_j)
			\leq \|Df_j\|(2\lambda B_j) + \|Df\|(2\lambda B_j) 
			\leq  2 \|Df\|(2\lambda B_j) + 2^{-j}\varepsilon.
		\end{equation}
		Moreover, note that $\sum_{j=1}^N (a_j-f)\varphi_j\to \sum_{j\in\mathbb N} (a_j-f)\varphi_j$ 
		in $L_{\loc}^1(\Omega)$ as $N\to\infty$, since the sum is locally finite.
		Below, we will also use the Leibniz rule from \cite[Theorem 1.2]{Lah20}.
		Since $\sum_{j\in\mathbb N}\varphi_j\equiv 1$ on $U$,
		we have
		\begin{align}
			&\|D(f_\varepsilon-f)\|(\Omega)\notag \\
			&\qquad = \left\| D\left( \sum_{j\in\mathbb N} (a_j-f)\varphi_j\right) \right\| (\Omega) \notag \\
			&\qquad  \leq \liminf_{N\to \infty}
			\left\| D\left( \sum_{j=1}^N (a_j-f)\varphi_j\right) \right\| (\Omega) 
			\quad \textrm{by lower semicontinuity}\notag \\
			&\qquad \overset{\eqref{eq:BV functions form vector space}}{\leq} \sum_{j\in\mathbb N} \left\| D\left(  (a_j-f)\varphi_j\right) \right\| (2B_j) \notag \\
			&\qquad \leq  \sum_{j\in\mathbb N}\left( \| D(a_j-f)\|(2B_j) + \int_{2B_j}|a_j-f|g_{\varphi_j}\,d\mu \right)\notag\\
			&\qquad \overset{\eqref{eq:aj minus f estimate}}{\le}
			\sum_{j\in\mathbb N}\left(2\|Df\|(2\lambda B_j)+2^{-j}\varepsilon\right)
			+  \sum_{j\in\mathbb N}
			\left( \int_{2B_j}|a_j-f_j|\frac{C_W}{r_j}\,d\mu + \int_{2B_j}|f_j-f|\frac{C_W}{r_j}\,d\mu \right);  \label{eq3.9-29Sep}
		\end{align}
		in the last inequality we also use the fact that $\varphi_j$ is $\frac{C_W}{r_j}$-Lipschitz.
		By the second property of \eqref{eq3.7-Sep}, for all $j\in\mathbb N$ we have
		\begin{equation}\label{eq3.12-Sep}
			\int_{2B_j}  |f_j-f|\frac{C_W}{r_j}\,d\mu \leq \frac{C_W}{r_j}  \frac{r_j}{2^j}\varepsilon
			=C_W  2^{-j}\varepsilon.
		\end{equation}
		By \eqref{eq3.8-Sep}, for all $j\in J_1$ we have
		\[
		\mu(\{x\in 2B_j: (f_j-M_j)_+>0\})\leq  \mu(2B_j\cap (G\cup N_j)) \le \mu(2B_j)/4,
		\]
		and so by Lemma \ref{lem2.2-Sep}, there is a constant $C_2>0$ such that
		\begin{align}
			\frac{C_W}{r_j}\int_{2B_j} (f_j-M_j)_+\,d\mu
			&\leq \frac{C_W}{r_j} C_2 r_j  {\|Df_j\|(2\lambda B_j)}\notag \\
			&\leq C_WC_2 \left(\|Df\|(2\lambda B_j)+2^{-j} \varepsilon\right)\label{eq3.13-Sep}
		\end{align}
		by the first property of \eqref{eq3.7-Sep}.
		In particular, for all $j\in\N$ we have that
		\begin{align}
			\int_{2B_j}|a_j-f_j|\frac{C_W}{r_j}\,d\mu 
			&\lesssim \|Df\|(2\lambda B_j)+2^{-j} \varepsilon. \label{eq3.18-7Oct}
		\end{align}
		Inserting the two estimates \eqref{eq3.12-Sep} and \eqref{eq3.18-7Oct} into
		\eqref{eq3.9-29Sep},
		we conclude from the bounded overlap property of $2\lambda B_j$
		that
		\begin{equation}\label{eq3.21-Oct}
			\|D(f_\varepsilon-f)\|(\Omega)
			\lesssim  \|Df\|(U)+\varepsilon.
		\end{equation}
		Combining this with \eqref{equ3.1-28Sep}, we obtain
		\[
		\|D(f_\varepsilon-f)\|(\Omega)
		\lesssim \varepsilon,
		\]
		which is the claim. The proof is completed.
	\end{proof}
	\begin{mybox}
		\begin{lemma}
			\label{lem3.3-Sep}
			We have $\|f_\varepsilon-f\|_{L^1(\Omega)}\lesssim \varepsilon$.
		\end{lemma}
	\end{mybox}
	\begin{proof}
		We have from the definition of $f_\varepsilon$ that
		\begin{align*}
			|f_\varepsilon-f| 
			&=  \left|\sum_{j\in\mathbb N} (a_j-f)\varphi_j\right| \chi_U\\
			&\leq  \sum_{j\in\mathbb N} (|a_j-f_j| + |f_j-f|)\chi_{2B_j}\\
			&= \sum_{j\in J_1} (f_j-M_j)_{+}\chi_{2B_j} + \sum_{j\in\mathbb N}|f_j-f|\chi_{2B_j}.
		\end{align*}
		By the the estimate \eqref{eq3.13-Sep} and the second property of \eqref{eq3.7-Sep}
		 (note that always $r_j\leq 1$), it follows that
		\begin{align}\notag
			\|f-f_\varepsilon\|_{L^1(\Omega)}
			&\leq \sum_{j\in J_1}\int_{2B_j}(f_j-M_j)_{+}\,d\mu
			+\sum_{j\in\mathbb N}\int_{2B_j}|f_j-f|\,d\mu \\
			&\leq \sum_{j\in\mathbb N} \bigg( C_2\left(\|Df\|(2\lambda B_j)
			+2^{-j}\varepsilon\right) +2^{-j} \varepsilon\bigg).
		\end{align}
		Combining this with the fact that $\Vert Df\Vert(U)<\varepsilon$
		from \eqref{equ3.1-28Sep}, we obtain the claim by the bounded overlap property of
		$2\lambda B_j$. The proof is completed.
	\end{proof}
	\begin{mybox}
		\begin{lemma}
			\label{lem3.5-Sep}$f^\vee_\varepsilon$ is upper semicontinuous on $\Omega$ and $f_\varepsilon^\wedge$ is lower semicontinuous on $\Omega$.
		\end{lemma}
	\end{mybox}
	\begin{proof}
		Since $f_\varepsilon\equiv f$ on $\Omega\setminus U$, and $f^\vee|_{\Omega\setminus G}$ and
		$f^\wedge|_{\Omega\setminus G}$ are respectively upper and lower
		semicontinuous by \eqref{equ3.2-28Sep}, we conclude that
		the functions $f^\vee_\varepsilon$ and
		$f^\wedge_\varepsilon$ are respectively upper and lower semicontinuous in
		$\Omega\setminus \overline{U}$.
		Since $f_\varepsilon$ is
		continuous on $U$, so are $f^\vee_\varepsilon$ and $f^\wedge_\varepsilon$.
		It remains to consider the case
		$x\in\partial U\cap\Omega$.
		
		Let $x\in \partial U\cap\Omega$ be arbitrary.
		We will first prove that $f^\vee_\varepsilon$ is upper semicontinuous at $x$.
		Let $\delta>0$ be arbitrary.
		Since $f^\vee$ is upper semicontinuous on $\Omega\setminus G$ by \eqref{equ3.2-28Sep},
		there exists $r_\delta>0$ with $B(x,2r_{\delta})\cap U\subset U_1$ and
		\begin{equation}\label{eq3.23-Oct}
			f^{\vee}(z)\le f^{\vee}(x)+\delta/2\quad \textrm{for all }z\in B(x,2r_\delta)\setminus G.
		\end{equation}
		Let $0<r<r_\delta$ be fixed.
		We have for $j\in J_1$ and for all $y\in 2B_j$ that
		\begin{align}
			a_j(y)
			&\le M_j\notag \\
			&={\rm esssup}_{z\in 2B_j\setminus (G\cup N_j)}f_j(z)\notag \\
			&\le {\rm esssup}_{z\in 2B_j\setminus (G\cup N_j)}f(z)+2^{-j} \label{eq3.20-28Oct}
		\end{align}
		by the third property of \eqref{eq3.7-Sep}.
		Let $y\in B(x,r)\cap U\subset B(x,2r)\cap U\subset U_1$.
		We have from the fact that $f_\varepsilon$ is continuous at $y$ and by Lemma \ref{lem3.4-Sep} that
		\begin{align*}
			f_\varepsilon^\vee(y)-f_\varepsilon^\vee(x)
			&= f_\varepsilon(y)-f^\vee(x) \\
			&=  \sum_{j\in J_1:\, 2B_j\ni y} (a_j(y)-f^\vee(x)) \varphi_j(y)\\
			&\overset{\eqref{eq3.20-28Oct}}{\leq}   \sum_{j\in J_1:\, 2B_j\ni y} \left[({\rm esssup}_{z\in 2B_j\setminus 
				(G\cup N_j)}f(z)-f^\vee(x))_{+} +2^{-j} \right]\varphi_j(y)\\
			&\leq  ({\rm esssup}_
			{z\in B(x,2r)\setminus G}f(z)-f^\vee(x))_{+} +\sum_{j\in J_1:\, 2B_j\cap B(x,2r)
				\neq \emptyset}2^{-j} \\
			&\overset{\eqref{eq3.23-Oct}}{\leq} \delta/2 +\sum_{j\in J_1:\, 2B_j\cap B(x,2r)
				\neq \emptyset}2^{-j} .
		\end{align*}
		By choosing $r>0$ smaller, if necessary, the sum above can also be made smaller
		than $\delta/2$, and so in total we have
		\[
		\sup_{B(x,r)}(f_\varepsilon^\vee-f_\varepsilon^\vee(x)) \le \delta.
		\]
		Thus  $f^\vee_\varepsilon$ is upper semicontinuous at $x$. 
		For the lower semicontinuity of $f^\wedge_\varepsilon$, the proof is similar.
		The proof is completed.

	\end{proof}
	
	\begin{proof}
		[Proof of Theorem \ref{thm1.1}]
		Let $\varepsilon>0$ be arbitrary. We take
		$U$ and $f_\varepsilon$ defined as in \eqref{equ3.1-28Sep} and \eqref{def-f}. Let $U_\varepsilon:=U$. Then
		
		(1) is given by  \eqref{equ3.1-28Sep};
		
		(2) follows since $\|f-f_\varepsilon\|_{\BV(\Omega)}<\varepsilon$ by Lemmas \ref{lem3.2-28Sep}-\ref{lem3.3-Sep};
		
		(3) is given by Lemma \ref{lem3.4-Sep};
		
		(4) is given by Lemma \ref{lem3.5-Sep}.
		
	\end{proof}
	
	\section{Proof of Corollary \ref{cor1.2}}
	
	Fix a point $O\in X$ and for each $i=0,1,\ldots$, define the annulus
	\[
	A_i:=\Omega\cap (B(O,i+1)\setminus B(O,i));
	\]
	we interpret $B(O,0)=\emptyset$.
	Also let $t\cdot A_i:= \{x\in \Omega: d(x,A_i)<(t-1)\}$ for $t>1$. We recall that $(X,d,\mu)$ is a complete doubling metric measure space supporting a $1$-Poincar\'e inequality.
	
	To prove Corollary \ref{cor1.2}, we need the following auxiliary lemmas.
	
	\begin{mybox}
		\begin{lemma}\label{lem4.2-Oct}
			Let $\Omega$ be an unbounded open subset of $X$, and let
			$f\in \BV(\Omega)$.
			Then for every $\varepsilon>0$, there exist an open set $E_\varepsilon\subset \Omega$
			with  ${\rm Cap}_1(E_\varepsilon)<\varepsilon$ and a function $h_\varepsilon$ so that the
			following properties hold:
			\begin{enumerate}
				\item $h_{\varepsilon}^\vee\equiv  f^\vee$
				and $h_{\varepsilon}^\wedge\equiv  f^\wedge$ on $ \Omega\setminus E_\varepsilon$;
				\item There is a sequence $\{b_i \}_{i=0}^{\infty}$ with $b_i$ is non-increasing, $b_i>0$
				and $\lim_{i\to\infty}b_{i}=0$ so that 
				\[
				|h_{\varepsilon}|\leq b_i {\rm\ \ on\ \ } 2\cdot A_i{\rm\ \ for\ all\ \ } i=0,1,\ldots;
				\]
				\item $\| h_{\varepsilon}-f\|_{\BV(\Omega)} \lesssim \|f\|_{\BV(E_\varepsilon)}< \varepsilon$.
			\end{enumerate}
		\end{lemma}
	\end{mybox}
	Note that Claim (2) simply amounts to saying that
	$$\lim_{\Omega\ni x\to\infty}h_\varepsilon^\vee(x)=0=\lim_{\Omega\ni x\to\infty}h_\varepsilon^\wedge(x);$$
	however, we will use the notation $b_i$ later.
	\begin{proof}
		Fix $\varepsilon>0$.
		Let $C'>0$ be the constant from \cite[Proof of Lemma 4.2]{LS17}, used below. Let $C_1$ be the constant from Lemma \ref{lem2.2-19Jul}.
		Then there is a sequence $d_i>0$ with $d_i$ is non-increasing so that
		\begin{equation}\label{eq3.21-Sep}
			\lim_{i\to\infty}d_i=0{\rm \ \ and\ \ } \sum_{i=0}^{\infty} 
			\frac{\|f\|_{\BV(3\cdot A_i)}}{d_i}<\frac{\varepsilon}{6C_1C'},
		\end{equation}
		because $\Vert f\Vert_{L^1(\Omega)}$ and $\|Df\|(\Omega)$ are both finite and because of 
		the bounded overlap property of $3\cdot A_i$.
		Let
		\[
		E_i:=\left\{x\in 2\cdot A_i: |f|^\vee(x) >d_i    \right\}
		\]
		for each $i\in\mathbb N$.
		
		We pick a partition of unity $\{ \eta_i\}_{i=0}^\infty$ with
		the following properties: $\sum_{i=0}^\infty\eta_i\equiv 1$ on $\Omega$, $0\leq \eta_i\leq 1$, each $\eta_i$ is a $C_W$-Lipschitz function for some $C_W\geq 1$,
		and 
		$\supp \eta_i\subset 2\cdot A_i$.
		Such a partition of unity exists, see e.g. the construction in \cite[Page 104-105]{HKST15}.

		We define the function $h_\varepsilon$ on $\Omega$ by setting
		\[
		h_{\varepsilon}:=\sum_{i=0}^\infty \max\{ -d_i, \min\{ d_i, f\}\}\eta_i.
		\]
		Then we let $b_{0}:=d_0$, $b_{1}:=d_0$, and 
		$b_i:=d_{i-2}$ for all $i=2,3,\ldots$. Recall that $d_i$ is non-increasing, and then so is
		$b_i$.
		Now we have that Claim (2) holds.
		Let $\phi_i$ be a $2$-Lipschitz function so that $0\leq \phi_i\leq 1$ on $\Omega$, $\phi_i\equiv 1$ on $2\cdot A_i$ and $\supp \phi_i\subset 3\cdot A_i$.
		By \cite[Proof of Lemma 4.2]{LS17} applied to the function $\phi_i f$, we obtain that 
		there is a finite constant $C'>0$ independent of $i\in\mathbb N$ so that for
		all $i=0,1,\ldots$,
		\[
		{\rm Cap}_1(E_i)
		\leq C' \frac{\|\phi_i f\|_{\BV(3\cdot A_i)}}{d_i}\leq 3C' \frac{\| f\|_{\BV(3\cdot A_i)}}{d_i},
		\]
		where the last inequality is obtained by the Leibniz rule \cite[Theorem 1.2]{Lah20}.
		Set $F_\varepsilon:=\bigcup_{i=0}^\infty E_i$.
		Combining this with \eqref{eq3.21-Sep}, we have 
		${\rm Cap}_1(F_\varepsilon)<\frac{\varepsilon}{2C_1}$. By Lemma \ref{lem2.2-19Jul}, there are a finite such constant $C_1>0$ and an open set $X\supset E_\varepsilon\supset F_\varepsilon$ so that ${\rm Cap}_1(E_\varepsilon)\leq C_1 {\rm Cap}_1(F_\varepsilon)+\varepsilon/2< \varepsilon$ and
		\[
		\lim_{r\to0}\frac{\mu(B(x,r)\cap F_\varepsilon)}{\mu(B(x,r))}=0 {\rm \ \ for \ all\ }x\in \Omega\setminus E_\varepsilon.
		\]
		By Lemma \ref{lem:variation measure and capacity}, we can also assume that
		$\Vert f\Vert_{\BV(E_\varepsilon)}<\varepsilon$. Also, we can assume that $E_\varepsilon\subset \Omega$. 
		It follows from the definition of $h_\varepsilon$ that $h_\varepsilon= f$ a.e.
		in $\Omega\setminus F_\varepsilon$,
		and hence we obtain for all $x\in \Omega\setminus E_\varepsilon$
		and for all $\delta>0$ that
		\begin{equation}\label{eq4.4-10Dec}
			\begin{split}
				& \lim_{r\to 0}\frac{\mu(B(x,r)\cap \{|f-h_\varepsilon|>\delta \})}{\mu(B(x,r))}\\
				& \quad \leq    \lim_{r\to 0}\frac{\mu(B(x,r)\cap \{y\in \Omega\setminus F_\varepsilon: |f(y)-h_{\varepsilon}(y)|>\delta\})}{\mu(B(x,r))} + \lim_{r\to0}\frac{\mu(B(x,r)\cap F_\varepsilon)}{\mu(B(x,r))}\\
				&\quad = 0.
			\end{split}
		\end{equation}
		Therefore, for all $t>0$, $\partial^*\{|f(x)-h_\varepsilon(x)|>t \}\cap (\Omega\setminus E_\varepsilon)=\emptyset$,
		and then from the coarea formula \eqref{eq2.8-Oct}
		and from \eqref{eq2.5-0411} it follows that
		\begin{equation}\label{eq:measure outside E eps}
			\Vert D(h_{\varepsilon}-f)\Vert(\Omega\setminus E_{\varepsilon})=0.
		\end{equation}
		
		By the arguments as in \eqref{eq:f epsilon f} and \eqref{eq3.13-3Dec},
		we  obtain from \eqref{eq4.4-10Dec} that $h_{\varepsilon}^\vee(x)= f^\vee(x)$
		and $h_{\varepsilon}^\wedge(x)= f^\wedge(x)$ for $x\in \Omega\setminus E_\varepsilon$.
		Now we have that Claim (1)  holds.
		
		Using \eqref{eq:measure outside E eps} and the Leibniz rule in
		\cite[Theorem 1.2]{Lah20}, we have that
		\begin{align}
			\|D(h_{\varepsilon}-f)\|(\Omega) 
			&=  \|D(h_{\varepsilon}-f)\|(E_\varepsilon)\notag\\
			&\leq  \|Df\|(E_\varepsilon)+ \|Dh_{\varepsilon}\|(E_\varepsilon) \notag \\
			&\leq  \|Df\|(E_\varepsilon)+  \sum_{i=0}^\infty  \|D (\max\{ -d_i, \min\{ d_i, 
			f\}\} \eta_i)\|(2\cdot A_i \cap E_\varepsilon) \notag   \\
			&\leq   \|Df\|(E_\varepsilon) + \sum_{i=0}^\infty  \left(\|D f\|(2\cdot A_i\cap E_\varepsilon) + C_W \int_{2\cdot A_i\cap E_\varepsilon}|f|\,d\mu \right) \notag  \\
			&\lesssim  \|Df\|({E_\varepsilon}) + \int_{E_\varepsilon}|f|\,d\mu \label{eq4.4-Oct}. \notag
		\end{align}
		Here we also used 
		the fact that $$\|D (\max\{ -d_i, \min\{ d_i, f\}\} ) \|(A)\leq \|Df\|(A)$$ for all
		$A\subset \Omega$.
		Thus Claim (3) also holds since $\|f\|_{\BV(E_\varepsilon)}<\varepsilon$ and $\|f-h_\varepsilon\|_{L^1(\Omega)}\leq  \|f\|_{L^1(E_\varepsilon)}$. The proof is completed.
	\end{proof}
	
	In Lemma \ref{lem4.2-Oct}, the conclusions $(1)$ and $(2)$ give us the limit at infinity outside a small set $E_\varepsilon$. 
	We  refer the interested reader to \cite{JKN23} for a discussion on limits at infinity
	outside a \emph{thin set} for $N^{1,p}$ functions on Ahlfors $Q$-regular metric measure spaces supporting
	a $p$-Poincar\'e inequality, where $1<Q<\infty$ and $1\leq p<\infty$. Furthermore, for radial
	and vertical
	limits at infinity on Muckenhoupt weighted Euclidean spaces, existence and uniqueness of limits at infinity
	along infinite paths for Sobolev functions have been studied in the recent works 
	\cite{EKN22, KN23, GKS24}.
	\begin{mybox}
		\begin{lemma}\label{lem4.4-21Oct}
			Let $f\in {\BV}(\Omega)$ so that $|f|\leq b_i$ on $2\cdot A_i$ for $i=0,1,2\ldots$,
			where $b_i$ is a sequence of
			positive numbers.
			Then the functions $f_j$ in \eqref{eq3.7-Sep}
			can be selected such that
			$|f_j|\leq b_i$ on $2 B_j\cap A_i $ when $2B_j\cap A_i\neq\emptyset$ and
			$2B_j\cap  (\bigcup_{k=0}^{i-1}A_k)=\emptyset$.
		\end{lemma}
	\end{mybox}
	\begin{proof} Let $f_j$ be a sequence in the construction \eqref{eq3.7-Sep}. We now define $f_j'$ by
		setting for  $2B_j\cap A_i\neq\emptyset$ and $2B_j\cap (\bigcup_{k=0}^{i-1}A_k)=\emptyset$ that
		\[
		f_j':=\max\{-b_i, \min\{ b_i, f_j\}\}.
		\]
		Note that $|f|\le b_i$ in $2B_j$.
		Therefore, the sequence $f_j'$ still satisfies \eqref{eq3.7-Sep}. The proof is completed.
	\end{proof}
	
	\begin{proof}
		[Proof of Corollary \ref{cor1.2}]
		Let $f\in \BV(\Omega)$ and let $\varepsilon>0$.
		By Lemma \ref{lem4.2-Oct},
		we find an open set $E_\varepsilon$ with
		${\rm Cap}_1(E_\varepsilon)<\varepsilon$ and
		a function $h_{\varepsilon}$ with $\|f-h_\varepsilon\|_{\BV(\Omega)}<\varepsilon/2$,
		$h_{\varepsilon}^\vee\equiv  f^\vee$,
		and $h_{\varepsilon}^\wedge\equiv  f^\wedge$ in $X\setminus E_\varepsilon$,
		and there is a sequence $\{ b_i\}_{i=0}^{\infty}$ with $b_i>0$ and $\lim_{i\to\infty}b_i=0$,
		and so that for all $i=0,1,\ldots$,
		$|h_{\varepsilon}|\leq b_i$ in $2\cdot A_i$.
		
		Then, starting from the function $h_{\varepsilon}$, we apply
		Theorem \ref{thm1.1} to obtain a function
		$f_\varepsilon$ and a subset $U_\varepsilon'$ of $\Omega$  defined as in \eqref{equ3.1-28Sep} and \eqref{def-f}.
		By the theorem, we can obtain ${\rm Cap}_1(U'_\varepsilon)<\varepsilon/2$,
		$\|h_\varepsilon-f_\varepsilon\|_{\BV(\Omega)}<\varepsilon/2$,
		$h_\varepsilon^\vee\equiv f_\varepsilon^\vee$ and
		$h_\varepsilon^\wedge\equiv f_\varepsilon^\wedge$ on $\Omega\setminus U'_\varepsilon$,
		and $f_\varepsilon^\vee$ is upper  semicontinuous and  $f_\varepsilon^\wedge$ is lower semicontinuous on $\Omega$.
		Defining $U_{\varepsilon}:=E_\varepsilon\cup U'_{\varepsilon}$, we have that
		$f_\varepsilon$ and $U_\varepsilon$ satisfy the properties given in Theorem \ref{thm1.1}.
		
		By Lemma \ref{lem4.4-21Oct}, we can choose the functions $f_j$ in \eqref{eq3.7-Sep} so that  
		$|f_j|\leq b_i$ on $A_i\cap 2 B_j$ when $2B_j\cap A_i\neq\emptyset$ and
		$2B_j\cap (\bigcup_{k=0}^{i-1}A_k)=\emptyset$.
		Then by the definition of $f_\varepsilon$ in \eqref{def-f} and 
		since $b_i\to0$ as $i\to\infty$, we obtain the  claim \eqref{eq1.2-15Nov}. The proof is completed.
	\end{proof}

	\section{Proof of Corollary \ref{cor1.4}}\label{sec:Euclidean}

	We begin by giving an example showing that the conclusion of Corollary \ref{cor1.4}
	does not hold in the metric measure space setting.
	
	\begin{example}\label{example5.4}
		We will construct a metric measure space $X$ on the plane. Let 
		\[
		{\bf O}:=(0,0),\ \ {\bf A}:=(0,1),\ \ {\bf C}:=(1,1)
		\]
		be points in $\mathbb R^2$. We set $X:=[{\bf O},{\bf A}]\cup[{\bf O},{\bf C}]$ where $[{\bf O},{\bf A}]$
		denotes the closed line segment connecting the points ${\bf O},{\bf A}$.
		We equip $X$ with
		the Euclidean distance $d_2$ and the $1$-dimensional Hausdorff measure $\mathcal H^1$. Then
		$(X,d_2,\mathcal H^1)$ is complete, $\mathcal H^1$ is doubling on $X$, and the space
		supports a $1$-Poincar\'e inequality.
		Let 
		\[
		f:=\chi_{[{\bf O},{\bf A}]}\in \BV(X).
		\]
		Then we have for every $x\in ({\bf O},{\bf A}]$, by taking $r:=\dist(x, [{\bf O},{\bf C}])$, 
		\[
		\mathcal Mf(x)\geq \dashint_{B(x,r)}f\,d\mathcal H^1=1.
		\]
		Hence $\mathcal Mf= 1$ on $({\bf O},{\bf A}]$. To prove that $\mathcal Mf$ is not continuous at ${\bf O}$, 
		it suffices to prove that $\mathcal Mf({\bf O})<1$. Let $B$ be an arbitrary open ball with center $x_B$
		and radius $r_B>0$, with ${\bf O}\in B$. Let
		$S_{\bf A}\in[{\bf O}, {\bf A}], S_{\bf C}\in [{\bf O},{\bf C}]$
		be the unique points such that $d_2(x_B, S_{\bf A})=d_2(x_B, S_{\bf C})=r_B$.
		
		Firstly, suppose that $x_B\in [{\bf O},{\bf C}]$. Then
		\[
		\mathcal H^1([{\bf O}, S_{\bf A}])\leq  \mathcal H^1([{\bf O}, S_{\bf C}]).
		\]
		This implies that
		\[
		\dashint_B f \,d\mathcal H^1
		= \frac{\mathcal H^1([{\bf O}, S_{\bf A}])}{ \mathcal H^1([{\bf O}, S_{\bf A}])+\mathcal H^1([{\bf O}, S_{\bf C}])}
		\leq \frac{1}{2}.
		\]
		
		Now we consider $x_B\in[{\bf O},{\bf A}]$. Let ${\bf M}$ be the middle point of $[{\bf O}, S_{\bf A}]$.
		The ball $B$ contains   the ball $B'$ with center at ${\bf M}$ and radius
		$r':=d_2({\bf O}, {\bf M})$.  Let $S'_{\bf C}\in[{\bf O},{\bf C}]$ be the point intersecting
		$\partial B'$.
		Since the triangle ${\bf O}{\bf M} S'_{\bf C}$ has a right angle at ${\bf M}$ and 
		the angle between
		$[{\bf O}, \bf{M}]$ and $[{\bf O}, S'_{\bf C}]$ is $\pi/4$, we have
		$d_2({\bf O}, S'_{\bf C})=\sqrt 2d_2({\bf O},{\bf M})=\sqrt 2 r'$.  We then have
		\[
		\dashint_B f \,d\mathcal H^1
		\leq \frac{2\mathcal H^1([{\bf O}, {\bf M}])}{ 2\mathcal H^1([{\bf O}, {\bf M}])+\mathcal H^1([{\bf O}, S'_{\bf C}])}= \frac{2r'}{2r'+\sqrt{2} r' }=\frac{2}{2+\sqrt 2}.
		\]
		We conclude that $\mathcal Mf({\bf O})<1$, and thus $\mathcal Mf$ is not continuous at ${\bf O}$.
		
		It is easy to check that every point in $(X,d_2,\mathcal H^1)$ has $1$-capacity at least $1$. 
		We conclude that the non-centered maximal function of a BV function may
		fail to be $1$-quasicontinuous in the metric space setting, contrary to the Euclidean setting.
		
		We also note that choosing $\varepsilon<1$, we have that $U_\varepsilon$ in Theorem \ref{thm1.1}
		is necessarily empty, and then $f=f_{\varepsilon}$ in $X$.
		In fact, $f^{\wedge}$ and $f^{\vee}$ are respectively lower and upper semicontinuous,
		but nonetheless $\mathcal Mf_{\varepsilon}=\mathcal Mf$ fails to be continuous. 
	\end{example}
	
	Due to this example, we formulate Corollary \ref{cor1.4} only in the Euclidean setting.
	
	\begin{proof}
		[Proof of Corollary \ref{cor1.4}]
		Let $f\in \BV(\Omega)$ be nonnegative and $\varepsilon>0$.
		Recall from \eqref{Sf minus Jf} that $S_f\setminus J_f$ has zero $1$-capacity.
		We let  $U$ and $G$ be subsets of $\Omega$ defined as in
		\eqref{equ3.1-28Sep}-\eqref{equ3.2-28Sep},
		so that ${\rm Cap}_1(U)<\varepsilon$, $f^\wedge|_{\Omega\setminus G}$  
		is  lower semicontinuous and $f^\vee|_{\Omega\setminus G}$  is upper semicontinuous,
		and we can further assume that $S_f\setminus J_f \subset G\subset U\subset \Omega$.
		Now,  defining $f_\varepsilon$ by   \eqref{def-f}, we obtain a function
		satisfying the properties given in Theorem \ref{thm1.1}.
		
		It remains to show that $\mathcal M_{\Omega}f_{\varepsilon}$ is continuous on $\Omega$.
		For this, we first prove a couple of auxiliary lemmas.
		\begin{mybox}
			\begin{lemma}\label{lem4.5-19Oct}
				Let $x\in\partial U\cap \Omega$. Then
				\begin{equation}\label{equ4.10-19Oct}
					\lim_{r\to0}\,\dashint_{B(x,r)}|f(y)-f_\varepsilon(y)|\,dy=0.
				\end{equation}
				In particular, we obtain that for every $\delta>0$, there is $r_\delta>0$ so that for all balls $B\subset B(x, r_\delta)$ with $x\in  \overline{B}$,
				\begin{equation}\label{equ4.11-19Oct}
					\dashint_{B}|f(y)-f_\varepsilon(y)|\,dy<\delta.
				\end{equation}
			\end{lemma}
		\end{mybox}
		\begin{proof}
			Recall the definitions from Section \ref{sec3}, 
			in particular the definition of $U_1$ from \eqref{equa3.6-Sep}.
			Let $\delta>0$. By \eqref{eq3.23-Oct}-\eqref{eq3.20-28Oct},
			there is $r_\delta>0$ with
			$B(x,2r_\delta)\cap U\subset U_1$ so that for every $j\in J_1$, we have
			\begin{equation} \label{eq4.8-19Oct}
				|f(y)-a_j(y)|\leq |f(y)|+M_j\leq f(y) + f^\vee(x) + \delta
			\end{equation}
			for every $y\in 2B_j$ with $2B_j\cap B(x,r_\delta)\neq\emptyset$;
			recall that $f$ is assumed to be nonnegative.
			Recall from \eqref{eq3.13-Oct} and \eqref{eq3.14-Sep} that 
			\begin{equation}\label{eq4.9-19Oct}
				\lim_{r\to0}\left(\frac{\mathcal L^d(B(x,r)\cap G)}{\mathcal L^d(B(x,r))}+ 
				\sum_{j\in J_1:\, 2B_j\cap B(x,r)\neq\emptyset}\frac{\mathcal L^d(N_j)}{\mathcal L^d(B(x,r))}\right)=0.
			\end{equation}
			Recall from \eqref{def-Leb},  \eqref{def-f+},  and \eqref{def-f-} that
			for all subsets $B\subset B(x,r)$ with
			$\mathcal L^d(B)/\mathcal L^d(B(x,r))$ bounded from below
			uniformly in $r$,
			\begin{equation}
				\begin{cases}
					\lim_{r\to0}\frac{1}{\mathcal L^d(B)}\int_{B\cap B^+_\nu(x,r)}|f(y)-f^+(x)|\,dy =0\quad\textrm{and}\\
					\lim_{r\to0}\frac{1}{\mathcal L^d(B)} \int_{B\cap B^-_\nu(x,r)}|f(y)-f^-(x)|\,dy  =0
					\quad\textrm{if }x\in J_f;\\
					\lim_{r\to0}\frac{1}{\mathcal L^d(B)} \int_{B}|f(y)-\widetilde{f}(x)|\,dy =0
					\quad\textrm{if }x\in \Omega\setminus S_f.
				\end{cases}\label{eq4.14-19Oct}
			\end{equation}
		Recall that the dilated Whitney balls $2B_j$ have overlap at most $C_0$. 
			It follows that 
			\begin{align}
				& \limsup_{r\to0}  \sum_{j\in J_1}\frac{1}{\mathcal L^d(B(x,r))}\int_{B(x,r)\cap 2B_j\cap (G\cup N_j)} |f(y)-a_j(y)|\,dy\notag \\ 
				&\quad\overset{\eqref{eq4.8-19Oct}}{\leq} \limsup_{r\to0}  \sum_{j\in J_1}\frac{1}{\mathcal L^d(B(x,r))}\int_{B(x,r)\cap 2B_j\cap (G\cup N_j)}f(y)\,d y\notag\\
				&\quad\quad + (f^\vee(x)+\delta) \limsup_{r\to 0}  \sum_{j\in J_1} \frac{\mathcal L^d(B(x,r)\cap 2B_j\cap (G\cup N_j))}{\mathcal L^d(B(x,r))}  \notag\\
				&\quad\leq  C_0 \limsup_{r\to0}  
				\frac{1}{\mathcal L^d(B(x,r))}\int_{B(x,r)}\bigg(\big(|f(y)-f^+(x)|\chi_{B_\nu^+(x,r)}(y)  \notag \\ &\quad\quad + |f(y)-f^-(x)|\chi_{B_\nu^-(x,r)}(y)\big)\chi_{J_f}(x)  + |f(y)-\widetilde{f}(x)|\chi_{\Omega\setminus S_f}(x)\bigg)\,dy\notag\\
				&\quad\quad\quad  + \bigg(3\max\{ f^+(x), f^-(x),\widetilde{f}(x)\} 
				+f^\vee(x)+\delta\bigg)\notag\\
				&\quad\quad\quad\quad \times \limsup_{r\to 0}  
				 \left(C_0\frac{\mathcal L^d(B(x,r)\cap G)}{\mathcal L^d(B(x,r))}+ C_0
				\sum_{j\in J_1:\, 2B_j\cap B(x,r)\neq\emptyset}\frac{\mathcal L^d(N_j)}{\mathcal L^d(B(x,r))}\right) \notag \\
				&\quad\quad = 0  \label{eq4.10-19Oct}
			\end{align}
			by  \eqref{eq4.9-19Oct} and \eqref{eq4.14-19Oct} with the choice $B=B(x,r)$.
			We then obtain
			\begin{align*}
				&\limsup_{r\to0}\dashint_{B(x,r)}|f(y)-f_\varepsilon(y)|\,dy\\
				&=\limsup_{r\to0}\int_{B(x,r)\cap U}\frac{|f(y)-f_\varepsilon(y)|}{\mathcal L^d(B(x,r))}\,dy\\
				&\le \limsup_{r\to0} \sum_{j\in J_1}\bigg(\int_{B(x,r)\cap 2B_j\setminus(G\cup N_j)}\frac{|f(y)-f_j(y)|}{\mathcal L^d(B(x,r))} \,dy + \int_{B(x,r)\cap 2B_j\cap(G\cup N_j)}\frac{|f(y)-a_j(y)|}{\mathcal L^d(B(x,r))} \,dy   \bigg)\\
				&\overset{\eqref{eq3.7-Sep}\textrm{ 3rd property } \&\ \eqref{eq4.10-19Oct}}{\leq}  \limsup_{r\to0}\sum_{j\in J_1: 2B_j\cap B(x,r)\neq\emptyset} 2^{-j} +0\\
				& = 0,
			\end{align*}
			which is the main claim \eqref{equ4.10-19Oct}. The last conclusion \eqref{equ4.11-19Oct} is obtained by the doubling property of $\mathcal L^d$ and the main claim \eqref{equ4.10-19Oct}.
			The proof is completed.
		\end{proof}
		\begin{mybox}
			\begin{lemma}\label{lem4.6-14Oct}
				We have $\mathcal M_{\Omega} f_\varepsilon\geq f_\varepsilon^\vee$ in $\Omega$.
			\end{lemma}
		\end{mybox}
		\begin{proof}
			Recall that $f$ is nonnegative.
			If $x\in U$, then $f_\varepsilon$ is continuous at $x$ and so
			$\mathcal M_{\Omega} f_\varepsilon (x) \geq \widetilde f_\varepsilon(x) = f_\varepsilon^\vee(x)$.
			If $x\in \Omega\setminus \overline{U}$, then
			by \eqref{equ4.6-13Oct} we have
			\[
			\mathcal M_{\Omega} f_\varepsilon(x)\geq \sup \dashint_{B_r}|f(y)|dy \overset{ \eqref{eq4.14-19Oct} }{\geq}   \max\{\widetilde{f}(x),f^+(x), f^-(x) \}=f^\vee(x)=f_\varepsilon^\vee(x)
			\]
			where the supremum is taken over all balls $B_r$ with radius
			$r< {\rm dist}(x, U)$, $x\in  \overline{B_r}$ and
			$B_r\subset \Omega\setminus \overline{U}$; here the first equality is obtained
			from \eqref{eq4.6-15Oct} and $S_f\setminus J_f\subset U$,
			and the last equality is obtained from the fact that
			$f^\vee=f^\vee_\varepsilon$ on $\Omega\setminus \overline{U}$.
			
			Finally, consider $x\in\partial U\cap \Omega$. Let $\delta>0$. 
			We have from \eqref{equ4.11-19Oct} in Lemma \ref{lem4.5-19Oct} that there is $r_\delta>0$ so that 
			\begin{equation}\label{eq4.16-19Oct}
				\sup_{B\in \mathcal T_1}\,\dashint_{B}{|f(y)-f_\varepsilon(y)|}\,dy<\delta.
			\end{equation}
			Here $\mathcal T_1$ is the collection of all balls $B\subset B(x, r_\delta)$ with $x\in  \overline{B}$.
			It follows that
			\begin{align}
				\mathcal M_{\Omega} f_\varepsilon (x) 
				&\geq  \sup_{B\in \mathcal T_1}\, \dashint_{B}f_\varepsilon(y)\,dy    \notag \\
				&\geq  \sup_{B\in \mathcal T_1}\, \dashint_{B} f(y)\,dy - \sup_{B\in \mathcal T_1}\,
				\dashint_{B} |f(y)-f_\varepsilon(y)|\, dy\notag \\
				&\overset{\eqref{eq4.16-19Oct}}{\geq}   \sup_{B\in\mathcal T_1}\, \dashint_{B}  f(y)\,dy -\delta\notag\\
				&\overset{\eqref{eq4.14-19Oct}}{\geq}  \max\{\widetilde{f}(x),f^+(x), f^-(x) \}-\delta\notag\\
				&\overset{\eqref{eq4.6-15Oct} }{=} f^\vee(x)-\delta\quad \textrm{since } S_f\setminus J_f\subset U \notag\\
				&=f_\varepsilon^\vee(x)-\delta\notag,
			\end{align}
			since $f^\vee=f^\vee_\varepsilon$ in $\Omega\setminus U$ by Lemma \ref{lem3.4-Sep}.
			Letting $\delta\to0$, we obtain the claim.
		\end{proof}
		\begin{mybox}
			\begin{lemma}\label{lem5.3-5Nov}
				Let $x\in \Omega$. If $\mathcal M_{\Omega}f_\varepsilon(x)\geq f^\vee_\varepsilon(x)$,
				then $\mathcal M_{\Omega}f_\varepsilon$ is upper semicontinuous at $x$.
			\end{lemma}
		\end{mybox}
		\begin{proof}
			Take an arbitrary sequence $x_j\to x$, $x_j\in \Omega$, such that
			\[
			\lim_{j\to\infty}\mathcal M_{\Omega} f_\varepsilon(x_j)=\limsup_{y\to x}
			\mathcal M_{\Omega} f_\varepsilon(y).
			\]
			We only need to show that
			$\mathcal M_{\Omega} f_\varepsilon(x)\ge \lim_{j\to\infty}\mathcal M_{\Omega} f_\varepsilon(x_j)$.
			
			We find ``almost optimal'' balls $B(x_j^*,r_j)$, such that
			\begin{equation}\label{eq:choice of almost optimal balls}\notag
				\lim_{j\to\infty}\mathcal M_{\Omega}f_\varepsilon(x_j)=\lim_{j\to\infty}\,\dashint_{B(x_j^*,r_j)}f_\varepsilon\,dy,
			\end{equation}
			with $x_j\in B(x_j^*,r_j)\subset\Omega$.
			Since $f_\varepsilon\in L^1(\Omega)$, we can assume that the radii $r_j$ are uniformly bounded.
			Now we consider two cases.
			
			\textbf{Case 1.} Suppose that by passing to a subsequence (not relabeled),
			we have $r_j\to 0$.
			We have
			\begin{equation}\label{eq4.18-19Oct}\notag
				\mathcal M_{\Omega} f_\varepsilon(x)\geq f^\vee_\varepsilon(x)\geq 
				\lim_{j\to\infty}\dashint_{B(x_j^*,r_j)}f_\varepsilon\,dy
			\end{equation}
			by the upper semicontinuity of $f^\vee_\varepsilon$ at $x$. Thus the conclusion follows.
			
			\textbf{Case 2.}
			The second possibility is that passing to a subsequence (not relabeled),
			we have $r_j\to r\in (0,\infty)$.
			Passing to a further subsequence (not relabeled),
			the vectors $x_j^*-x_j$ (since they have length at most $r_j)$
			converge to some $v\in\R^d$.
			Now for $x^*:=x+v$ we have
			$\chi_{B(x_j^*,r_j)}\to \chi_{B(x^*,r)}$ in $L^1(\R^d)$, with
			$x\in \overline{B}(x^*,r)$ and $B(x^*,r)\subset \Omega$.
			In this case we have by \eqref{equ4.6-13Oct} that
			\[
			\mathcal M_{\Omega} f_\varepsilon(x)\ge \dashint_{B(x^*,r)}f_\varepsilon\,dy
			=\lim_{j\to\infty}\,\dashint_{B(x_j^*,r_j)}f_\varepsilon\,dy
			=\lim_{j\to\infty}\mathcal M_{\Omega} f_\varepsilon(x_j).
			\]
			This completes the proof.	
		\end{proof}
		
		[Proof of Corollary \ref{cor1.4} continued] Notice from the definition of the non-centered
		maximal function that $\mathcal M_{\Omega} f_\varepsilon$ is obviously lower semicontinuous in $\Omega$.
		By Lemma \ref{lem4.6-14Oct} and Lemma \ref{lem5.3-5Nov}, we obtain that
		$\mathcal M_{\Omega} f_\varepsilon$ is upper semicontinuous in
		$\Omega$. Hence $\mathcal M_{\Omega} f_\varepsilon$ is continuous in $\Omega$,
		which is the last claim.
		The proof is completed.
	\end{proof}

	\vspace{1cm}
	
	{\bf Data availability}: No data was used for the research described in this article.
	
	\newcommand{\etalchar}[1]{$^{#1}$}
	
\end{document}